\documentclass[11pt]{amsart}
\usepackage{setspace,multirow}
\usepackage{amssymb}
\usepackage{amsmath, amsthm,mathrsfs, bbm}
\usepackage[margin=1in]{geometry}
\usepackage{fancyhdr}
\usepackage{enumerate}
\usepackage{dsfont}
\usepackage{graphicx}
\usepackage{float}
\usepackage[%
    font={small,sf},
    labelfont=bf,
    format=hang,    
    format=plain,
    margin=0pt,
    width=0.8\textwidth,
]{caption}
\usepackage[list=true]{subcaption}

\usepackage{listings}
\usepackage{color} 
\definecolor{mygreen}{RGB}{28,172,0} 
\definecolor{mylilas}{RGB}{170,55,241}
\usepackage{hyperref}
\newtheorem{theorem}{Theorem}[section]
\newtheorem{lemma}[theorem]{Lemma}
\newtheorem{proposition}[theorem]{Proposition}

\newtheorem{conjecture}[theorem]{Conjecture}

\theoremstyle{definition}
\newtheorem{definition}[theorem]{Definition}

\theoremstyle{remark}
\newtheorem{remark}[theorem]{Remark}

\numberwithin{equation}{section}

\title[Stochastic stabilization in $\mathbb{C}^2$]{Stabilization by Noise of a $\mathbb{C}^2$-Valued Coupled System}
\author[Chen]{Joe P. Chen}
\address{\textbf{Joe P. Chen} \\ Department of Mathematics, Colgate University, Hamilton, NY 13346, USA.}
\email{jpchen@colgate.edu}

\author[Ford]{Lance Ford}
\address{\textbf{Lance Ford} \\Department of Biostatistics and Epidemiology, University of Oklahoma Health Sciences Center, Oklahoma City, OK 73126, USA.}
\email{lance-ford@ouhsc.edu}

\author[Kielty]{Derek Kielty}
\address{\textbf{Derek Kielty} \\Department of Mathematics, University of Illinois, Urbana, IL 61801, USA.}
\email{dkielty2@illinois.edu}

\author[Majumdar]{Rajeshwari Majumdar}
\address{\textbf{Rajeshwari Majumdar}\\ Department of Mathematics, University of Connecticut, Storrs, CT 06269, USA.}
\email{rajeshwari.majumdar@uconn.edu}

\author[McCain]{Heather McCain}
\address{\textbf{Heather McCain} \\Division of Marine Science, The University of Southern Mississippi, Hattiesburg, MS 39406, USA.}
\email{heather.mccain@usm.edu}

\author[O'Connell]{Dylan O'Connell}
\address{\textbf{Dylan O'Connell} \\Department of Statistics, Yale University, New Haven CT 06511, USA.}
\email{dylan.oconnell@yale.edu}

\author[Shum]{Fan Ny Shum}
\address{\textbf{Fan Ny Shum}\\ Department of Mathematics, Courant Institute of Mathematical Sciences, New York University, New York, NY 10012, USA.}
\email{fanny.shum@nyu.edu}

\thanks{This work is supported by the Research Experience for Undergraduates (REU) program at the Department of Mathematics, University of Connecticut in Summer 2015, under the NSF grant DMS-1262929 (PI: Luke Rogers), and the NSF grant DMS-1405169 (PI: Maria Gordina).}

\date{\today}

\begin{document}

\maketitle
\begin{abstract}
D. Herzog and J. Mattingly have shown that a $\mathbb{C}$-valued polynomial ODE with finite-time blow-up solutions may be stabilized by the addition of $\mathbb{C}$-valued Brownian noise. In this paper, we extend their results to a $\mathbb{C}^2$-valued system of coupled ODEs with finite-time blow-up solutions. We show analytically and numerically that stabilization can be achieved in our setting by adding a suitable Brownian noise, and that the resulting system of SDEs is ergodic. The proof uses the Girsanov theorem to induce a time change from our $\mathbb{C}^2$-system to a quasi-$\mathbb{C}$-system similar to the one studied by Herzog and Mattingly.
\end{abstract}

\tableofcontents 

\section{Introduction} 
\label{sec:intro}
An \textit{explosive} system is a system of differential equations with trajectories that blow up in finite time. Some systems of this type have been shown to be stable by adding a random noise; that is, by adding a small amount of randomness transversal to an explosive trajectory, it pushes the trajectory onto a dynamically stable path. While there are examples where the addition of noise does not guarantee stability (see Scheutzow's construction in \cite{scheutzow}), usually one can stabilize an explosive system by adding a suitable Brownian noise.

This idea of stabilization was influenced by the study of turbulence. Constantin, Lax, and Majda developed a simple 1-dimensional model for the 3-dimensional vorticity equation of incompressible fluid flow in \cite{vorticity}. They proved that the system $\dot{z}=z^2$ for $z\in \mathbb{C}$, a particular example of their vorticity model, blows up in finite time. More references on explosive models can be found in \cite{vorticity}. In \cite{bec, bch, ghw}, the flow of certain fluids can be modeled by an SDE with a polynomial drift term. This inspired Herzog and Mattingly \cite{2011,noise1} to study the stability of the complex-valued SDE
\begin{equation}\label{eq:poly}
dz_t = (a_{n+1}z_t^{n+1} + a_n z_t^n + ... +a_0)\,dt + \sigma\, dB_t
              \end{equation}
with initial condition $z_0\in \mathbb{C}$, where $B_t=B_t^{(1)}+iB_t^{(2)}$, $B_t^{(1)}$ and $B_t^{(2)}$ are independent real-valued standard Brownian motions, and $\sigma \in \mathbb{R^+}$. They showed that the SDE (\ref{eq:poly}) has a solution for all finite times and initial conditions, and its solution possesses a unique invariant measure. In particular, they showed the system (\ref{eq:poly}) is
\textit{ergodic}; roughly speaking, it has the same behavior averaged over time as averaged over the space of all the systemÕs states.

To observe this stabilization phenomenon, consider the SDE
\begin{equation}\label{eq:zn}
		dz_t = z_t^{n+1}\,dt + \sigma \,dB_t.
              \end{equation}      
When $\sigma=0$, equation (\ref{eq:zn}) has $n$ explosive trajectories, each of which lies along the ray ${\rm arg}(z)=2\pi k/n$, $k\in \{0,1,\cdots, n-1\}$. As shown in Figure \ref{fig:zsqphase}, for $n=1$, there is one explosive trajectory, when $z=$ Re$(z)>0$. Intuitively, we need just enough noise to perturb the explosive trajectory onto a stable path. Notice in Figure \ref{fig:zsqtraj}, for the initial condition $z_0=2$, we have an explosive trajectory (indicated in red) when $\sigma=0$ and a nonexplosive trajectory (indicated in blue) when $\sigma\neq 0$. In addition, the blue trajectory follows a similar pattern shown in Figure \ref{fig:zsqphase}. Hence, when $\sigma\neq 0$, we have a stabilized system.


\begin{figure}
\centering
\subcaptionbox[Short Subcaption]{%
    Phase Portrait%
    \label{fig:zsqphase}%
}
[%
    0.45\textwidth 
]%
{%
    \includegraphics[width=0.45\textwidth]%
    {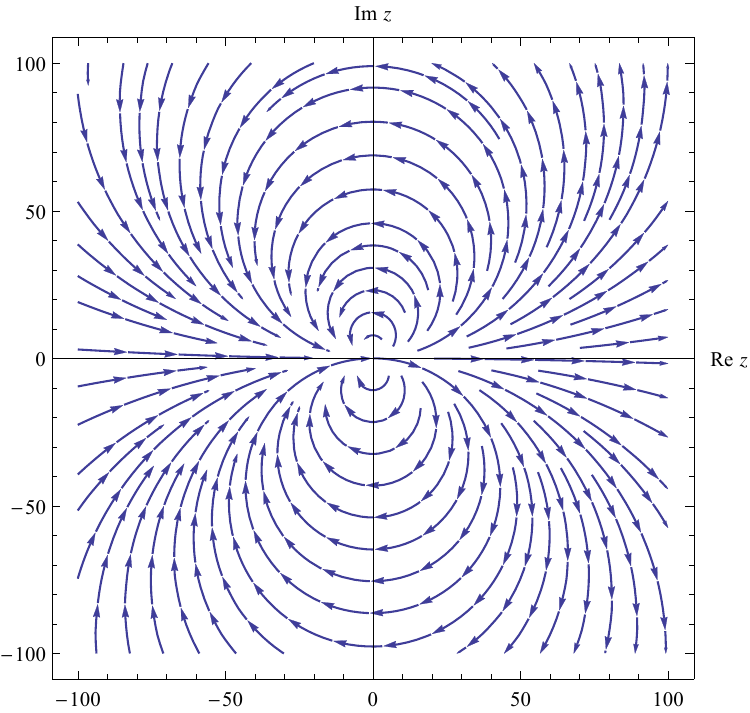}%
}%
\hspace{0.05\textwidth} 
\subcaptionbox[Short Subcaption]{%
    Trajectories from the initial condition $z_0=2$%
    \label{fig:zsqtraj}%
}
[%
    0.45\textwidth 
]%
{%
    \includegraphics[width=0.47\textwidth]%
    {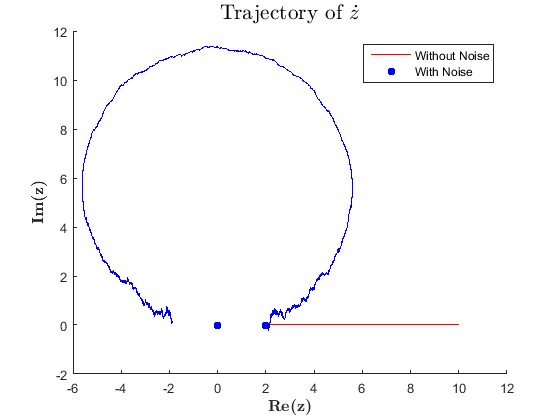}%
}%
\caption[Short Caption]{Example of $\dot{z}=z^2$}
\end{figure}


From the physics perspective, equations (\ref{eq:poly}) and (\ref{eq:zn}) are interesting because they correspond to complex-valued Langevin equations arising from minimization of certain complex-valued energy functionals, also known as \emph{path integrals} (but with non-positive-definite integrands). Making sense of these path integrals will lead to better understanding of lattice gauge quantum chromodynamics (QCD) models. Recently, G. Aarts and collaborators have studied special cases of (\ref{eq:poly}) and (\ref{eq:zn}) from the numerical point-of-view, \emph{cf.\@} \cite{Aarts1, Aarts2} and references therein. They have numerically calculated the invariant measure of the SDE, and obtained an approximate spectrum of the infinitesimal generator of the SDE. However, there is still work to do to bridge the gap between the stochastic analysis (ergodicity, exponential convergence to equilibrium) and the numerical calculations (spectral simulations, physics implications) for these SDEs.

The primary goal of this paper is to extend (\ref{eq:poly}) and (\ref{eq:zn}) to the multivariate setting. Herzog and Mattingly established a criterion for the existence of a strictly positive transition density of a degenerate diffusion in \cite{transition}; specifically, one that satisfies SDE (\ref{eq:poly}) and similar SDEs of higher dimensions. One particular example examined in \cite{transition} is the Galerkin truncations of a randomly forced viscous Burgers' equation. The Burgers' equation is a fundamental PDE occurring in fluid mechanics. Herzog and Mattingly provided a sufficient condition for the existence of a positive transition density for this Burgers' equation with additive noise.

We would like to extend the stability condition established in \cite{2011,noise1,noise2} to the stochastics Burgers' equation. To do so, we begin with a simple model of the Burger's equation stated in \cite{transition}. Consider the $\mathbb{C}^2$-valued system of ODEs 
  \begin{equation}\label{eq:1}
    \left\{
                \begin{array}{l}
                     \dot{z}_t = -\nu z_t + \alpha z_tw_t \\
		\dot{w}_t = -\nu w_t + \beta z_tw_t
                \end{array}
     \right.
  \end{equation}
with initial condition $(z_0, w_0)\in \mathbb{C}^2$. Here $\nu\in \mathbb{R}^+$ and $\alpha, \beta\in\mathbb{R}$. This system has a pair of fixed points: a sink at the origin and a saddle point at $(\nu/\beta, \nu/\alpha)$. Trajectories which lie on the unstable manifold associated with the saddle point, but not in the basin of attraction near the origin, will blow up in finite time. So the question is: Which types of complex-valued Brownian motions can one add to stabilize these explosive trajectories? And if so, can one prove ergodicity of the system?

We show that the answer is positive in at least one scenario. To state the result, let us fix some definitions. 
Let $(\Omega, \mathcal{F},\mathcal{F}_t,P)$ be a probability space equipped with a filtration $\{\mathcal{F}_t\}_{t\geq 0}$ and a probability measure $P$. Then let $(X_t)_{t\geq 0}$ be a $\mathbb{R}^d$-valued stochastic process adapted to the filtration $\{\mathcal{F}_t\}_{t\geq 0}$. Denote by $P_x$ the law of $X_t$ starting at $x\in \mathbb{R}^d$.

\begin{definition}
\label{def:explosiontime}
For each fixed $n\geq 0$, let $\xi_n=\inf\{t\geq 0: |X_t|\geq n\}$ be the exit time of $X_t$ from the ball of radius $n$ centered at the origin. Then the \textit{explosion time} $\xi$ of the process $X_t$ is defined as $\xi:=\sup_n\xi_n$.
\end{definition}

\begin{definition}
\label{def:nonexplosive}
A process $X_t$ is said to be \textit{nonexplosive} if $P_x(\xi<\infty)=0$ for all $x\in\mathbb{R}^d$.
\end{definition}

Our main result is:

\begin{theorem}
\label{thm:main}
Consider the system of SDEs
\begin{align}\label{eq:MainSDE}
  \left\{
                \begin{array}{l}
                     dz_t = \left(-\nu z_t + \alpha z_tw_t \right)\,dt+ \sigma\,dB_t \\
		dw_t = \left(-\nu w_t + \beta z_tw_t\right)\,dt + \frac{\beta}{\alpha} \sigma \,dB_t
                \end{array}
     \right.
\end{align}
with initial condition $X_0 = (z_0, w_0) \in \mathbb{C}^2$, where $\nu \in \mathbb{R}^+$, $\alpha,\beta\in \mathbb{R}\setminus \{0\}$, $\sigma \in \mathbb{R}\setminus \{0\}$, and $B_t = B_t^{(1)} + iB_t^{(2)}$ is a $\mathbb{C}$-valued standard Brownian motion.
Then the solution of (\ref{eq:MainSDE}) is nonexplosive, and moreover, possesses a unique ergodic (invariant) measure.
\end{theorem}

In fact, we can say more about the behavior of the invariant measure. To do so, we first introduce the notation
\begin{align}
x_1={\rm Re}(z), \quad x_2={\rm Im}(z), \quad x_3={\rm Re}(w), \quad x_4= {\rm Im}(w).
\end{align}
Then let
\begin{align}
                  y_1=\frac{1}{2}\left(x_1+\frac{\alpha}{\beta}x_3\right),~
                 y_2=\frac{1}{2}\left(x_1-\frac{\alpha}{\beta}x_3\right),~
                 y_3=\frac{1}{2}\left(x_2+\frac{\alpha}{\beta}x_4\right),~
                 y_4=\frac{1}{2}\left(x_2-\frac{\alpha}{\beta}x_4\right).
\end{align}
Finally, let $\tilde{z} = y_1 + iy_3$ and $\tilde{w} = y_2+ iy_4$.

\begin{proposition}[Invariant measure]
\label{prop:main}
The system (\ref{eq:MainSDE}) has the unique invariant measure $\pi(\tilde{z})\delta_0(\tilde{w})$, where $\pi$ is the unique invariant measure for the $\mathbb{C}$-valued system
\begin{align}
\label{eq:z20}
d\tilde{z}_t=(-\nu \tilde{z}_t+\beta\tilde{z}_t^2)\,dt+\sigma\,dB_t,
\end{align}
and $\delta_0$ is the delta measure at $0$. 
\end{proposition}
Note that (\ref{eq:z20}) is of the form (\ref{eq:poly}) with $n=1$, $a_2=\beta$, $a_1=-\nu$, and $a_0=0$. As mentioned previously, Herzog and Mattingly \cite{noise1} have established the stability of (\ref{eq:poly}).

The proofs of Theorem \ref{thm:main} and Proposition \ref{prop:main} involve a change of coordinates, which reduces the $\mathbb{C}^2$-system (\ref{eq:MainSDE}) to a quasi-$\mathbb{C}$-system similar to (\ref{eq:z20}), and an application of the Girsanov theorem. We also present numerical evidence that supports Theorem \ref{thm:main} (in the case that an isotropic Brownian noise is added), as well as the case where an anisotropic Brownian noise is added.

The layout of this paper is as follows. In \S\ref{sec:analysis}, we describe the linear transformation which reduces our deterministic $\mathbb{C}^2$-valued ODEs to a deterministic quasi-$\mathbb{C}$-valued ODE. We also perform a dynamical analysis to identify the explosive regions of the deterministic system. This serves as preparation for \S\ref{sec:ergo}, where we add an isotropic Brownian noise and show rigorously the reduction of our $\mathbb{C}^2$-valued SDEs (\ref{eq:MainSDE}) to a quasi-$\mathbb{C}$-valued SDE similar to (\ref{eq:z20}). From this we can deduce the ergodic properties of our system (\ref{eq:MainSDE}) \emph{\`{a} la} Herzog, and thus prove Theorem \ref{thm:main} and Proposition \ref{prop:main}. In \S\ref{sec:numerical}, we give numerical evidence for the stabilization of our $\mathbb{C}^2$ system by adding an isotropic or an anisotropic Brownian noise. Some concluding remarks are given in \S\ref{sec:conclusion}.

\subsubsection*{Acknowledgements}
The problem studied in this paper was proposed by David P. Herzog, to whom we are grateful for his help and guidance with the project. JPC and FNS also thank Iddo Ben-Ari, Jonathan Mattingly, and Alexander Teplyaev for useful discussions, and especially Maria Gordina for a careful reading of the manuscript and helpful critiques.

\section{Analytical observations and heuristics}\label{sec:analysis}
\subsection{Reduction via a change of coordinates}\label{sec:reduce}
We begin by rewriting the system (\ref{eq:1}) in terms of the coordinate ${\bf x}= (x_1, x_2, x_3, x_4)$, where 
 \begin{align}
x_1={\rm Re}(z), \quad x_2={\rm Im}(z), \quad x_3={\rm Re}(w), \quad x_4= {\rm Im}(w).
 \end{align}
This results in the following system of equations:
\begin{align}
\label{eq:ODENew}
    \left\{
                \begin{array}{ll}
                    \dot{x}_1 =-\nu x_1 + \alpha (x_1x_3-x_2x_4) \\
		\dot{x}_2 =-\nu x_2 + \alpha (x_2x_3+x_1x_4) \\
		\dot{x}_3 =-\nu x_3 + \beta (x_1x_3-x_2x_4) \\
		\dot{x}_4 =-\nu x_4 + \beta (x_2x_3+x_1x_4), \\
                \end{array}
     \right.
\end{align}
with initial condition $(x_1(0), x_2(0), x_3(0), x_4(0))$. In what follows, we shall assume that $\alpha>0$ and $\beta>0$ without loss of generality. (If one of $\alpha$ or $\beta$ vanishes, the system degenerates. For all other cases, one can replace some of the $x_i$ by $-x_i$ to effectively make $\alpha$ and $\beta$ positive.)

 \begin{figure}
 \centering
 \includegraphics[width=0.45\textwidth]{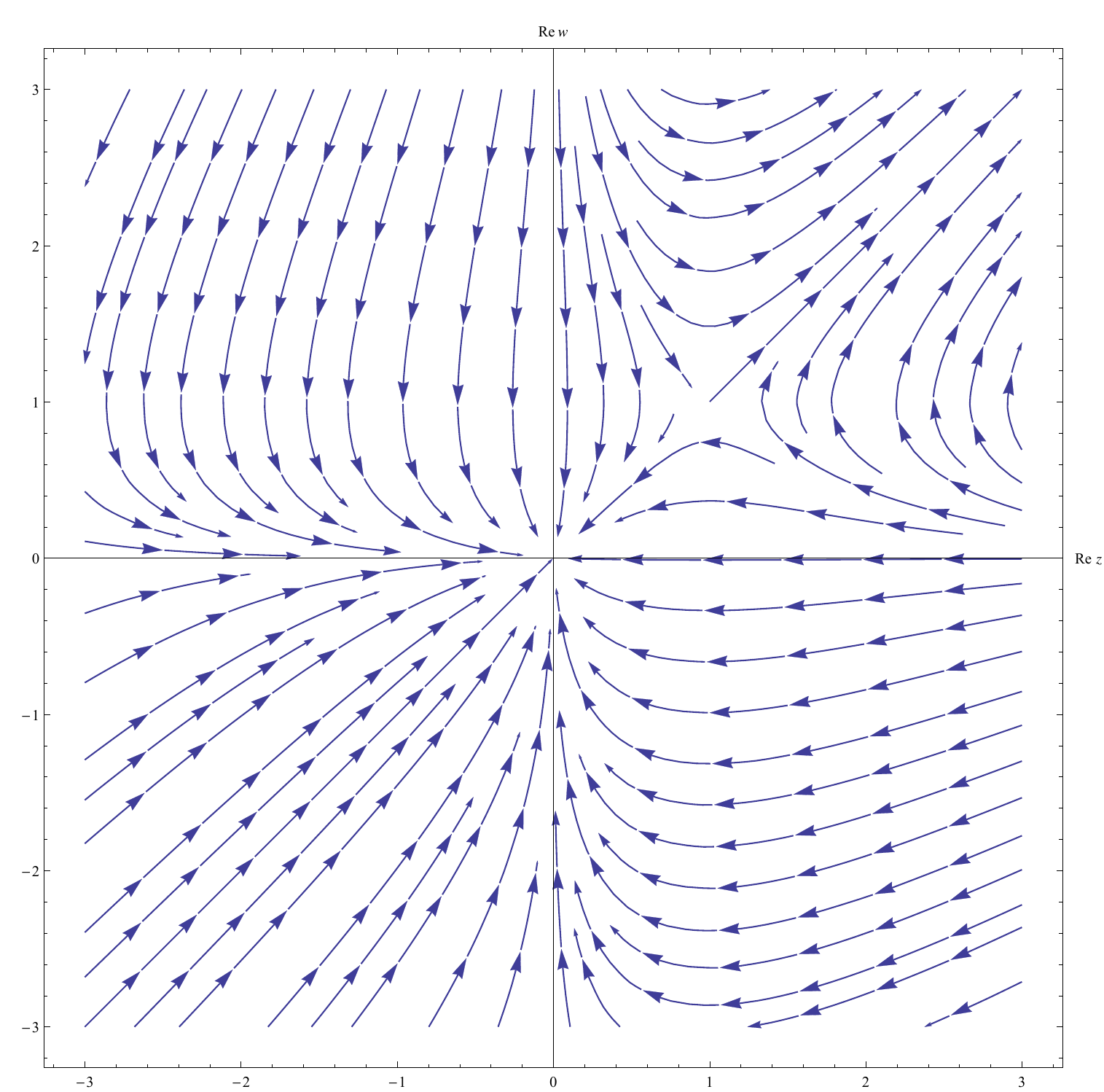}
 \caption{Phase portrait of the system (\ref{eq:ODENew}) restricted to the $x_1 x_3$-plane.}
 \label{fig:real}
\end{figure}

By setting the time derivatives to $0$, we readily find the two fixed points of the system, ${\bf 0}=(0,0,0,0)$ and ${\bf p}=(\nu/\beta, 0, \nu/\alpha, 0)$. We then linearize (\ref{eq:ODENew}) about each of the fixed points:
\begin{align}
\label{jac1}
\begin{bmatrix}\dot{x}_1 \\ \dot{x}_2 \\ \dot{x}_3 \\ \dot{x}_4\end{bmatrix} = \begin{bmatrix} -\nu & 0 & 0 & 0 \\ 0 & -\nu & 0 & 0 \\ 0 & 0 & -\nu & 0 \\ 0 & 0 & 0 & -\nu \end{bmatrix} \begin{bmatrix}x_1\\x_2\\x_3\\x_4\end{bmatrix} + O({\bf x}^2) ,
\end{align} 
\begin{align}
\label{jac2}
\begin{bmatrix}\dot{x}_1 \\ \dot{x}_2 \\ \dot{x}_3 \\ \dot{x}_4\end{bmatrix} = \begin{bmatrix} 0 & 0 & \nu\alpha/\beta & 0 \\ 0 & 0 & 0 & \nu\alpha/\beta \\ \nu\beta/\alpha & 0 & 0  & 0 \\ 0 & \nu \beta/\alpha & 0 & 0 \end{bmatrix} \begin{bmatrix}x_1-\nu/\beta \\x_2\\x_3-\nu/\alpha\\x_4\end{bmatrix} + O(({\bf x}-{\bf p})^2) .
\end{align}
It is clear from the Jacobian in (\ref{jac1}) that ${\bf 0}$ is an attracting fixed point. On the other hand, the Jacobian in (\ref{jac2}) has eigensolutions
\begin{align}
\label{eigdxn}
&\lambda_1= -\nu, \quad {\bf e}_1 = \begin{bmatrix} 0 \\  -\alpha/\beta \\ 0 \\ 1\end{bmatrix}; \quad
\lambda_2= -\nu, \quad {\bf e}_2 = \begin{bmatrix}  -\alpha/\beta \\ 0 \\ 1 \\ 0\end{bmatrix}; \\
\nonumber &\lambda_3= +\nu, \quad {\bf e}_3 = \begin{bmatrix} 0 \\  +\alpha/\beta \\ 0 \\ 1\end{bmatrix}; \quad
\lambda_4= +\nu, \quad {\bf e}_4 = \begin{bmatrix} +\alpha/\beta \\ 0 \\ 1 \\ 0\end{bmatrix}.
\end{align} 
This implies that a $2$-dimensional unstable manifold and a $2$-dimensional stable manifold are associated to the saddle point ${\bf p}$. See Figure \ref{fig:real} for the phase portrait.

The above linearization analysis shows that $\nu$ and $\alpha/\beta$ are the only figures of merit for the system (\ref{eq:ODENew}). Moreover, the symmetry of the eigensolutions suggests that the effective dynamics may be simpler than the $4$-dimensional nature of the system. To this end, let us make a change of coordinates so that the new coordinate directions agree with the eigendirections in (\ref{eigdxn}):
\begin{equation}
\label{eq:coordchange}
                  y_1=\frac{1}{2}\left(x_1+\frac{\alpha}{\beta}x_3\right),~
                 y_2=\frac{1}{2}\left(x_1-\frac{\alpha}{\beta}x_3\right),~
                 y_3=\frac{1}{2}\left(x_2+\frac{\alpha}{\beta}x_4\right),~
                 y_4=\frac{1}{2}\left(x_2-\frac{\alpha}{\beta}x_4\right).
\end{equation} 

By replacing the $x_i$ by the $y_i$, we rewrite (\ref{eq:1}) as
\begin{equation}\label{eq:5}
    \left\{
                \begin{array}{l}
                  \dot{y_1}=-\nu y_1+\beta\left[(y_1^2-y_2^2)-(y_3^2-y_4^2)\right]\\
                 \dot{y_2}=-\nu y_2\\
                 \dot{y_3}=-\nu y_3+2\beta(y_1y_3-y_2y_4)\\
                 \dot{y_4}=-\nu y_4,\\
                \end{array}
              \right. 
\end{equation} 
with initial condition $(y_1(0), y_2(0), y_3(0), y_4(0))$. Observe that $y_2$ and $y_4$ evolve autonomously under (\ref{eq:5}), with solutions
\begin{align}
y_2(t) = y_2(0) e^{-\nu t}, \quad y_4(t) = y_4(0) e^{-\nu t}.
\end{align}
Plugging these back into (\ref{eq:5}) yields the following $2$-dimensional system:
\begin{equation}
\label{eq:6}
    \left\{
                \begin{array}{l}
                  \dot{y_1}(t)=-\nu y_1(t)+\beta\left(y_1^2(t)-y_3^2(t)\right)-\beta\left(y_2^2(0)-y_4^2(0)\right)e^{-2\nu t}\\
                 \dot{y_3}(t)=-\nu y_3(t)+2\beta y_1(t)y_3(t)-2\beta y_2(0)y_4(0)e^{-2\nu t}
                \end{array}
              \right..
\end{equation} 

Notice that by setting $\tilde{z} = y_1+i y_3$, and without the exponentially decaying terms of order $e^{-2\nu t}$, (\ref{eq:6})
resembles the system
\begin{equation}\label{eq:7}
\dot{\tilde{z}}=-\nu\tilde{z}+\beta\tilde{z}^2,
\end{equation}
the stochastic stabilization of which was studied by Herzog and Mattingly \cite{noise1}. So on a heuristic level, our stabilization problem in $\mathbb{C}^2$,
\begin{align}\label{eq:orig}
\left\{\begin{array}{l}
dz_t = \left(-\nu z_t + \alpha z_t w_t\right)\,dt + \text{Brownian noise} \\ 
dw_t = \left(-\nu w_t + \beta z_t w_t\right)\,dt + \text{Brownian noise}
\end{array} \right.,
\end{align}
can be reduced to the stabilization problem in $\mathbb{C}$,
\begin{align}\label{eq:new}
d\tilde{z}_t = \left(-\nu \tilde{z}_t + \beta \tilde{z}_t^2\right)\,dt + \text{Brownian noise}.
\end{align}
In \S\ref{sec:ergo}, we will make this heuristic rigorous for a class of Brownian noises.


\subsection{Conditions for (non)explosivity of deterministic solutions}\label{sec:explosive}

Analyzing (\ref{eq:6}) for sets of initial conditions corresponding to explosive solutions is now more tractable, since (\ref{eq:1}) has been reduced to a system evolving over $\mathbb{R}^2$ rather than $\mathbb{R}^4$. The following proposition gives us estimates on the boundaries of the explosive regions.

Denote ${\bf y}(t) := (y_1(t),y_2(t),y_3(t),y_4(t))$ and its initial condition ${\bf y}_0:=(y_1(0),y_2(0),y_3(0),y_4(0))$.\\ 
Let $I_{\rm max}$ be the largest interval $[0,T)$ on which ${\bf y}(t)$ is defined. By the existence and uniqueness theorem for first-order ODEs, $I_{\rm max}\in (0,\infty]$. Also, since the RHS of (\ref{eq:6}) is real analytic, the solutions to (\ref{eq:6}) are also real analytic. It suffices to show that they are continuously differentiable. 

We now state a sufficient condition for explosivity.
\begin{proposition} \label{prop1}

The solution ${\bf y}(t)$ of (\ref{eq:6}) with initial condition ${\bf y}_0$ is explosive if all of the following conditions hold:
\begin{enumerate}[(I)]
\item $y_3(0)=0$.\label{itm:1}
\item Either $y_2(0)=0$ or $y_4(0)=0$.\label{itm:2}
\item $y_1(0)>C$ for some large enough constant $C$ which depends on $\beta, \nu, y_2(0), y_4(0)$.\label{itm:3}
\end{enumerate}
\end{proposition}

\begin{remark}
\label{rem:y3always0}
Conditions (\ref{itm:1}) and (\ref{itm:2}) imply that $y_3(t) = 0$ for all $t \in I_{\rm max}$.
\end{remark}

\begin{proof}[Proof of Proposition \ref{prop1}]
Let $y_3(0)=0$.
Suppose $y_2(0) = 0$, so that (\ref{eq:6}) reduces to
\begin{equation}
                 \dot{y_1}(t)=-\nu y_1(t)+\beta y_1^2(t)+ \beta y_4^2(0)e^{-2\nu t},\\
\label{eq:8}
\end{equation} 
with initial condition $y_1(0)$. The solutions of (\ref{eq:8}) are bounded below by the solutions of
\begin{equation}
    \left\{
                \begin{array}{ll}
                  \dot{x}(t)=-\nu x(t)+\beta x^2(t)\\
		x(0) = y_1(0)
                \end{array}
              \right.,
\label{eq:9}
\end{equation} 
since $\dot{y}_1(t)>\dot{x}(t)$ for all times $t$. It is easy to verify that solutions of (\ref{eq:9}) are explosive whenever $y_1(0)>\frac{\nu}{\beta}$. Therefore, solutions of (\ref{eq:8}) are explosive under the same initial condition $y_1(0)>\frac{\nu}{\beta}$.

Next suppose $y_4(0) = 0$. Then (\ref{eq:6}) reduces to
\begin{equation}
                 \dot{y_1}(t)=-\nu y_1(t)+\beta y_1^2(t) - \beta y_2^2(0)e^{-2\nu t},
\label{eq:10}
\end{equation} 
with initial condition $y_1(0)$. Similar to the arguments given above, we see that the solutions of (\ref{eq:10}) are bounded below by the solutions of
\begin{equation}
    \left\{
                \begin{array}{ll}
                  \dot{x}(t)=-\nu x(t)+\beta x^2(t)-  \beta y_2^2(0) \\
		x(0) = y_1(0)
                \end{array}
              \right.,
\end{equation} 
which explode whenever $y_1(0)>\frac{1}{2}\Big{(}\frac{\nu}{\beta}+\sqrt{(\frac{\nu}{\beta})^2+(2y_2(0))^2}\Big{)}$.
\end{proof}


\begin{figure}
\centering
\subcaptionbox[Short Subcaption]{%
    $y_2(0)=y_3(0)=0$%
    \label{fig:4}%
}
[%
    0.45\textwidth 
]%
{%
    \includegraphics[width=0.45\textwidth]%
    {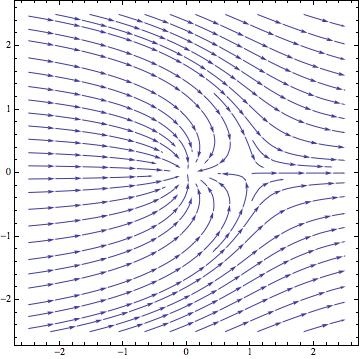}%
}%
\hspace{0.05\textwidth} 
\subcaptionbox[Short Subcaption]{%
    $y_3(0)=y_4(0)=0$%
    \label{fig:5}%
}
[%
    0.45\textwidth 
]%
{%
    \includegraphics[width=0.45\textwidth]%
    {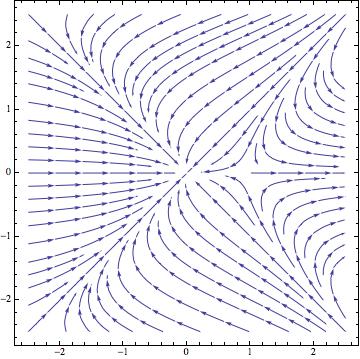}%
}%
\caption[Short Caption]{Phase Portraits of $\dot{\bf y}(t)$ for $\beta=\nu=1$ }
\end{figure}

\begin{remark} From simulations it appears difficult to analytically pin down the constant $C$ for the initial condition $y_2(0)=y_3(0)=0$ (see Figure \ref{fig:4}). On the other hand, for the initial condition $y_3(0)=y_4(0)=0$, the estimates on $C$ are at least qualitatively correct (see Figure \ref{fig:5}). 
\end{remark}

\begin{remark}
Based on our analysis and simulations, we believe the converse of Proposition \ref{prop1} will hold true as well. In particular, based on the dynamics, if $y_3(0)\neq 0$ and for certain initial conditions, if either $y_2(0)\neq 0$ or $y_4(0)\neq 0$, then it cannot blow up in finite time. Unfortunately, we have yet to prove this explicitly; hence we state this result as a conjecture below.
\end{remark}

\begin{conjecture}\label{prop2}
If conditions (\ref{itm:1}) or (\ref{itm:2}) of Proposition \ref{prop1} does not hold, then ${\bf y}(t)$ is nonexplosive. 
\end{conjecture}

\section{Nonexplosivity and ergodicity of the $\mathbb{C}^2$-valued SDEs} \label{sec:ergo}

In this section we make rigorous the heuristics stated towards the end of \S\ref{sec:reduce}, and prove Theorem \ref{thm:main} and Proposition \ref{prop:main}. Recall that our objective is to add a complex-valued Brownian noise to stabilize the deterministic system (\ref{eq:1}). For the proofs, we will assume that the Brownian noise is of the form $(\sigma \,B_t, \frac{\beta}{\alpha}\sigma \,B_t)$ in the $(z,w)$-coordinates, where $\sigma\in \mathbb{R}\setminus \{0\}$ and $B_t$ is a complex-valued standard Brownian motion. (Note that $B_t$ is the same Brownian motion in both coordinates.) The corresponding SDE is (\ref{eq:MainSDE}). A direct calculation shows that (\ref{eq:MainSDE}) can be rewritten as
\begin{align}
\label{y1y3}
\left\{\begin{array}{l}
  d y_1(t)=(-\nu y_1(t)+\beta\left(y_1^2(t)-y_3^2(t)\right)-\beta\left(y_2^2(0)-y_4^2(0)\right)e^{-2\nu t})\, dt + \sigma\, dB_t^{(1)}\\
                 d y_3(t)=(-\nu y_3(t)+2\beta y_1(t)y_3(t)-2\beta y_2(0)y_4(0)e^{-2\nu t})\, dt + \sigma\,dB_t^{(2)}
 \end{array}\right.
\end{align}
in the `reduced' coordinates ${\bf y}(t)=(y_1(t), y_2(t), y_3(t), y_4(t))$ defined in (\ref{eq:coordchange}). Here $B_t^{(1)}$ and $B_t^{(2)}$ are independent \emph{real}-valued standard Brownian motions, and $B_t = B_t^{(1)}+ i B_t^{(2)}$. Furthermore, if we define $\tilde{z}_t= y_1(t) + iy_3(t)$ and $\tilde{w}_0= y_2(0)+iy_4(0)$, then $\tilde{z}_t$ satisfies the SDE
\begin{equation}\label{eq:ztilde}
	d\tilde{z}_t=(-\nu \tilde{z}_t+\beta\tilde{z}_t^2-\beta \tilde{w}_0^2e^{-2\nu t})\,dt+\sigma\,dB_t.
\end{equation}
Observe that, without the term of order $e^{-2\nu t}$, if $\nu=0$, then (\ref{eq:ztilde}) is a special case of the $\mathbb{C}$-valued SDE with polynomial drift (\ref{eq:poly}), studied in \cite{noise1}.

We now state several related results from \cite{noise1} which will be used in the sequel \cite{noise2}. 

\begin{proposition}[\cite{noise1}]
\label{prop:HMresults}
Consider the complex-valued system
\begin{align*}
                \begin{array}{l}
                     dz_t = \left(a_{n+1} z_t^{n+1} + a_n z_t^n + \cdots + a_0 \right)\,dt+ \sigma\,dB_t 
                \end{array}
\end{align*}
with initial condition $z_0 \in \mathbb{C}$, where $n\geq 1$ is an integer, $a_i \in \mathbb{C}$, $a_{n+1}\neq 0$, $\sigma \in \mathbb{R}\setminus \{0\}$, and $B_t = B_t^{(1)} + iB_t^{(2)}$ is a $\mathbb{C}$-valued standard Brownian motion. The following hold for the process $z_t$:
\begin{enumerate}
\item $z_t$ is nonexplosive in the sense of Definition \ref{def:nonexplosive}.
\item $z_t$ possesses an ergodic invariant probability measure $\pi$, which is absolutely continuous with respect to the Lebesgue measure on $\mathbb{C}$.
\end{enumerate}
\end{proposition}
\begin{proof}
The approach of \cite{noise1} is to construct a Lyapunov function pair corresponding to the It\^{o} process $z_t$; see Theorem 5.2 of \cite{noise1} for the precise statement, and Sections 6 and 7 of \cite{noise1} for the proofs. Once this Lyapunov pair is constructed, then one can apply the general results from Section 4 of \cite{noise1} to the process $z_t$. Item (1) is stated in the proof of Lemma 4.4 of \cite{noise1}, and is given in Appendix A of \cite{noise1}. Item (2) is Theorem 4.5(a) of \cite{noise1} applied to $z_t$, as one can verify that the infinitesimal generator $\mathcal{L}$ is uniformly elliptic. 
\end{proof}

We recall that a process $z_t$ is \emph{ergodic} with respect to the measure $\pi$ if for all $f\in L^1(\pi)$,
\begin{align}
\label{eq:ergo}
P\left(\lim_{t\to\infty} \frac{1}{t} \int_0^t\, f(z_s)\,ds = \int_{\mathbb{C}} \, f(z) \,d\pi(z) \right)=1,
\end{align}
\emph{cf.\@} Chapter 4 of \cite{hasminskii}. Applying Lebesgue's dominated convergence theorem, (\ref{eq:ergo}) implies that for all $z_0 \in\mathbb{C}$ and all bounded functions $f$,
\begin{align}
\lim_{t\to\infty}\frac{1}{t} \int_0^t\, E^P_{z_0}[f(z_s)]\,ds = \int_{\mathbb{C}}\, f(z)\,d\pi(z).
\end{align}
In particular, the invariant measure $\pi$ can be obtained as the limit
\begin{align}
\label{eq:ergo2}
\pi(\cdot) = \lim_{t\to\infty} \frac{1}{t}\int_0^t\, P_{z_0}(z_s\in \cdot)\,ds,
\end{align}
which is independent of $z_0 \in \mathbb{C}$.

It is also of interest to characterize the rate of exponential convergence to the ergodic measure. See Theorem 3.3 of \cite{noise1} for results concerning the process (\ref{eq:poly}). We will not pursue the rate of convergence question for the process $\tilde{z}_t$ in this paper.

In any case, the key step is to justify the connection between (\ref{eq:poly}) and (\ref{eq:ztilde}) so that the ergodic properties of the former can be transferred to the latter. This will be achieved using the Girsanov transform.

\begin{lemma}[Girsanov transform]
\label{lem:timechange}
Let $B_t$ be a $\mathbb{C}$-valued standard Brownian motion on $(\Omega, \mathcal{F}, \mathcal{F}_t, P)$, and $z_t$ and $\tilde{z}_t$ be It\^o processes of respective forms
\begin{align}
\label{eq:z} 
dz_t = \left(-\nu z_t + \beta z_t^2\right)\,dt+ \sigma \,dB_t,\\
d\tilde{z}_t =(-\nu \tilde{z}_t+\beta\tilde{z}_t^2-\beta \tilde{w}_0^2e^{-2\nu t})\,dt+\sigma\,dB_t,
\end{align}
both of which have the same initial condition $z_0=\tilde{z}_0 \in \mathbb{C}$. For each $t\in (0,\infty)$, let
\begin{align}
\label{theta}  \theta(t) &= -\frac{\beta \tilde{w}_0^2}{\sigma} e^{-2\nu t},\\
M_t &= \exp\left(- \int_0^t\, {\rm Re}[\theta(s)] \,dB_s^{(1)} - \int_0^t\, {\rm Im}[\theta(s)]\,dB_s^{(2)}- \frac{1}{2} \int_0^t\,  |\theta(s)|^2\,ds \right),\\
dQ_t &= M_t \,dP \quad \text{on}~\mathcal{F}_t,\\
\widehat{B}_t &= \int_0^t \, \theta(s)\,ds+ B_t.
\end{align}
Then:
\begin{enumerate}
\item $\{M_t: t\geq 0\}$ is a uniformly integrable martingale.
\item There exists a probability measure $Q$ on $\mathcal{F}_\infty$ such that $Q|_{\mathcal{F}_t} = Q_t$. Moreover $P$ and $Q$ are equivalent measures.
\item $\widehat{B}_t$ is a $\mathbb{C}$-valued standard Brownian motion under $Q$.
\item The $Q$-law of $\tilde{z}_t$ is the same as the $P$-law of $z_t$ for all $t\in [0,\infty]$.
\end{enumerate} 
\end{lemma}

\begin{proof}
From standard SDE theory, we know that the Girsanov transform from $\tilde{z}_t$ to $z_t$ holds on the time interval $[0,T]$ for some finite $T>0$ if Novikov's condition,
\begin{equation}\label{eq:nov}
E^P\left[\exp\left(\frac{1}{2}\int_0^T\, |\theta(s)|^2 \,ds\right)\right]<\infty,
\end{equation}
is satisfied. 
By (\ref{theta}),
\begin{align*}
E^P\left[\exp\left(\frac{1}{2}\int_0^T\, |\theta(s)|^2 \,ds\right)\right]=E^{P}\left[\exp\left(\frac{1}{2} \beta^2|\tilde{w}_0|^4\int_0^T \,e^{-4\nu s}\,ds\right)\right].
\end{align*}
Since 
\begin{align*}
|\tilde{w}_0|^4 = |y_2(0)+iy_4(0)|^4 =[(y_2(0))^2+(y_4(0))^2]^2 \leq 2 [(y_2(0))^4 + (y_4(0))^4],
\end{align*}
we get
\begin{align*}
\frac{1}{2}\beta^2|\tilde{w}_0|^4\int_0^T\,e^{-4\nu s}\,ds &\leq \beta^2 [(y_2(0))^4 + (y_4(0))^4] \int_0^T\,e^{-4\nu s}\,ds\\
&= \beta^2\left[(y_2(0))^4+(y_4(0))^4\right]\frac{1-e^{-4\nu T}}{4\nu} < \infty.
\end{align*}
This verifies Novikov's condition (\ref{eq:nov}).

In order to extend the Girsanov transform to $T=\infty$, we need to verify that the martingale $\{M_t: t\geq 0\}$ is uniformly integrable, \emph{cf.\@} Item (1) of the Lemma. By the preceding calculation, we see that there exists a finite constant $C$ (taken to be $\frac{\beta^2}{4\nu}[(y_2(0))^4 + (y_4(0))^4]$) such that for all $t>0$, the first moment of $M_t$ satisfies
\begin{align}
\label{eq:UI1}
E^P[M_t] = E^P\left[ \exp\left(\frac{1}{2}\int_0^t\, |\theta(s)|^2 \,ds\right)\right] \leq C.
\end{align}
Meanwhile, the second moment of $M_t$ satisfies
\begin{align*}
E^P[M_t^2] &= E^P\left[\exp\left( - \int_0^t\, 2{\rm Re}[\theta(s)] \,dB_s^{(1)} - \int_0^t\, 2{\rm Im}[\theta(s)]\,dB_s^{(2)}- \int_0^t\,  |\theta(s)|^2\,ds\right)\right]\\
&= \exp\left( \int_0^t\, 2 |\theta(s)|^2\,ds - \int_0^t\, |\theta(s)|^2\,ds \right) = \exp\left(\int_0^t\, |\theta(s)|^2\,ds \right) \leq C^2.
\end{align*}
Applying the Cauchy-Schwarz inequality, we obtain that for any measurable subset $A$ of $\mathbb{C}$,
\begin{align}
\label{eq:UI2}
E^P [|M_t| \mathbbm{1}_A] \leq \left(E^P[|M_t|^2]\right)^{1/2} \left(E^P[\mathbbm{1}_A]\right)^{1/2} \leq C [P(A)]^{1/2}.
\end{align}
The estimates (\ref{eq:UI1}) and (\ref{eq:UI2}) together imply that $\{M_t\}$ is uniformly integrable. Items (2) through (4) of the Lemma now follow from  Proposition VIII.1.1, Proposition VIII.1.1', and Theorem VIII.1.4 of \cite{RevuzYor} (see also Proposition VIII.1.15 of \cite{RevuzYor} for the statement of Novikov's condition on the time interval $[0,\infty]$).
\end{proof}

\begin{proposition}
\label{prop:nonexplosive}
$\tilde{z}_t$ is nonexplosive in the sense of Definition \ref{def:nonexplosive}.
\end{proposition}
\begin{proof}
Let $\xi$ (resp.\@ $\tilde\xi$) be the explosion time of $z_t$ (resp.\@ $\tilde{z}_t$) as defined in Definition \ref{def:explosiontime}. By items (2) and (4) of Lemma \ref{lem:timechange}, we have the equivalence
\begin{align*}
P_{z_0}(\tilde\xi <\infty) =  0  ~\Longleftrightarrow~ Q_{z_0}(\tilde\xi <\infty)=0  ~\Longleftrightarrow~ P_{z_0}(\xi<\infty)=0.
\end{align*}
Since $P_{z_0}(\xi <\infty)=0$ for all $z_0 \in \mathbb{C}$ by Item (1) of Proposition \ref{prop:HMresults}, we deduce that $P_{z_0}(\tilde\xi <\infty)=0$ for all $z_0 \in \mathbb{C}$. This proves the nonexplosivity of (\ref{eq:ztilde}). 
\end{proof}

The ensuing computation allows us to identify the limiting distribution of the process $\tilde{z}_t$.

\begin{lemma}
\label{lem:uniqueinv}
Suppose (\ref{eq:ergo2}) holds. Then for each $z_0\in \mathbb{C}$,
\begin{align}
\label{equallaw}
\lim_{t\to\infty} \frac{1}{t} \int_0^t \, P_{z_0}(\tilde{z}_s \in \cdot)\,ds = \pi(\cdot),
\end{align}
where $\pi$ is as in (\ref{eq:ergo2}).
\end{lemma}

\begin{proof}
To begin, we fix $r\in (0,t)$ and use the Girsanov transform to write
\begin{align}
\label{eq:avg}
\frac{1}{t} \int_0^t\, P_{z_0}(\tilde{z}_s\in \cdot)\,ds 
&= \frac{1}{t} \int_0^t\, E^Q_{z_0}\left[\mathds{1}_{\{\tilde{z}_s\in \cdot\}} M_s^{-1}\right]\,ds\\
\nonumber &= \frac{1}{t} \int_0^r \, E^Q_{z_0}\left[\mathds{1}_{\{\tilde{z}_s\in \cdot\}} M_s^{-1}\right]\,ds + \frac{1}{t} \int_r^t \, E^Q_{z_0}\left[\mathds{1}_{\{\tilde{z}_s\in \cdot\}} M_s^{-1}\right]\,ds\\
\nonumber &=  \frac{1}{t} \int_0^r \, E^Q_{z_0}\left[\mathds{1}_{\{\tilde{z}_s\in \cdot\}} M_s^{-1}\right]\,ds +\frac{1}{t} \int_r^t \, E^Q_{z_0}\left[\mathds{1}_{\{\tilde{z}_s\in \cdot\}} M_r^{-1}\right] \,ds \\ 
\nonumber & \quad+ \frac{1}{t} \int_r^t \, E^Q_{z_0}\left[\mathds{1}_{\{\tilde{z}_s\in \cdot\}} M_r^{-1} (R(r,s)-1)\right]\,ds,
\end{align}
where
\begin{align*}
M_s^{-1} &:= \exp\left(\int_0^s\, {\rm Re}[\theta(\xi)] \,dB_\xi^{(1)} + \int_0^s\, {\rm Im}[\theta(\xi)]\,dB_\xi^{(2)} - \frac{1}{2}\int_0^s\, |\theta(\xi)|^2 \,d \xi  \right),\\
R(r,s) &:= \exp\left(\int_r^s \,{\rm Re}[\theta(\xi)]\, dB_\xi^{(1)} + \int_r^s \,{\rm Im}[\theta(\xi)]\,dB_\xi^{(2)} - \frac{1}{2}\int_r^s\, |\theta(\xi)|^2\,d\xi\right).
\end{align*}
We denote the three integrals in the RHS of (\ref{eq:avg}) by $I_1$, $I_2$, and $I_3$, respectively. To complete the proof, we will show that in the limit $t\to\infty$ followed by $r\to\infty$, $I_1 \to 0$, $I_2 \to \pi(\cdot)$, and $I_3\to 0$.

First of all, using the fact that $M_s^{-1}$ is a mean-$1$ martingale, we get
\begin{align*}
\varlimsup_{t\to\infty} |I_1| \leq \varlimsup_{t\to\infty} \frac{1}{t} \int_0^r \, E^Q_{z_0}[M_s^{-1}]\,ds =\varlimsup_{t\to\infty} \frac{1}{t}\int_0^r \,1\,ds =\varlimsup_{t\to\infty} \frac{r}{t}=0.
\end{align*}
Next, using the Markov property of $\tilde{z}_t$, a change of variables, and Tonelli's theorem, we can write
\begin{align*}
I_2 &= \frac{1}{t} \int_r^t \, E^Q_{z_0}\left[ E^Q_{\tilde{z}_r}\left[ \mathds{1}_{\{\tilde{z}_{s-r}\in \cdot\}}\right] M_r^{-1}\right]\,ds \\
&= \frac{t-r}{t} \cdot \frac{1}{t-r} \int_0^{t-r} E^Q_{z_0} \left[ Q_{\tilde{z}_r}(\tilde{z}_s\in \cdot) M_r^{-1}\right]\,ds\\
&= \frac{t-r}{t} \cdot E^Q_{z_0}\left[\left(\frac{1}{t-r}\int_0^{t-r} Q_{\tilde{z}_r}(\tilde{z}_s\in \cdot)\,ds \right) M_r^{-1}\right].
\end{align*}
Recall that the $Q$-law of $\tilde{z}_t$ is equal to the $P$-law of $z_t$, and the definition of $\pi$ in (\ref{eq:ergo2}). By Reverse Fatou's lemma,
\begin{align*}
\varlimsup_{t\to\infty} I_2 &\leq \left(\varlimsup_{t\to\infty} \frac{t-r}{t}\right) \cdot E^Q_{z_0} \left[\varlimsup_{t\to\infty} \left(\frac{1}{t-r}\int_0^{t-r} Q_{\tilde{z}_r}(\tilde{z}_s\in \cdot)\,ds \right) M_r^{-1} \right]= E^Q_{z_0} \left[\pi(\cdot) M_r^{-1}\right].
\end{align*}
Similarly, by Fatou's lemma,
\begin{align*}
\varliminf_{t\to\infty} I_2 \geq E^Q_{z_0}\left[\pi(\cdot) M_r^{-1}\right]. 
\end{align*}
Since $M_r^{-1}$ is a uniformly integrable martingale, it follows that
\begin{align*}
\lim_{r\to\infty} \lim_{t\to\infty} I_2 = \pi(\cdot) E^Q_{z_0}[M_\infty^{-1}] = \pi(\cdot).
\end{align*}

Finally, for $I_3$ we apply the Cauchy-Schwarz inequality twice, first with respect to the $Q$-expectation and then with respect to the $s$-integral, to find
\begin{align*}
|I_3| &\leq \frac{t-r}{t} \cdot \frac{1}{t-r} \int_r^t\, \left(E^Q_{z_0}[M_r^{-2}]\right)^{1/2} \left( E^Q_{z_0}|R(r,s)-1|^2\right)^{1/2}\,ds \\
&\leq \frac{t-r}{t}\cdot (E^Q_{z_0}[M_r^{-2}])^{1/2} \cdot \left( \frac{1}{t-r}\int_r^t\, E^Q_{z_0}|R(r,s)-1|^2\,ds\right)^{1/2}.
\end{align*}
Note that
\begin{align*}
&~\quad E^Q_{z_0}[M_r^{-2} ] \\
&= E^Q_{z_0}\left[\exp\left(\int_0^r\,{\rm Re}[ 2\theta(\xi)]\,dB_\xi^{(1)} +\int_0^r\,{\rm Im}[ 2\theta(\xi)]\,dB_\xi^{(2)} - \frac{1}{2}\int_0^r \, |2\theta(\xi)|^2\,d\xi + \int_0^r |\theta(\xi)|^2\,d\xi\right) \right]\\
&= \exp\left(\int_0^r\, |\theta(\xi)|^2\,d\xi\right),
\end{align*}
and $E_{z_0}^Q\left[M_\infty^{-2}\right] <\infty$ because $\theta \in L^2([0,\infty])$. An analogous calculation gives that $E^Q_{z_0}[R(r,s)]=1$ and 
\begin{align*}
E^Q_{z_0} |R(r,s)-1|^2 = E^Q_{z_0}[R(r,s)]^2 -1
= \exp\left(\int_r^s \,|\theta(\xi)|^2\,d\xi \right)-1.
\end{align*}
Thus 
\begin{align*}
\varlimsup_{t\to\infty}  \frac{1}{t-r}\int_r^t\, E^Q_{z_0}|R(r,s)-1|^2\,ds &= \varlimsup_{t\to\infty} \left[\frac{1}{t-r}\int_r^t \, \exp\left(\int_r^s\, |\theta(\xi)|^2\,d\xi\right)\right]-1\\
&\leq \varlimsup_{t\to\infty} \exp\left(\int_r^t\, |\theta(\xi)|^2\,d\xi\right) -1\\
&= \exp\left(\int_r^\infty\, |\theta(\xi)|^2\,d\xi\right)-1.
\end{align*}
Putting everything together we get
\begin{align*}
\varlimsup_{r\to\infty} \varlimsup_{t\to\infty} |I_3| \leq \varlimsup_{r\to\infty} \left[\exp\left(\int_0^r \,|\theta(\xi)|^2\,d\xi\right)\right]^{1/2}\cdot \varlimsup_{r\to\infty} \left[ \exp\left(\int_r^\infty\, |\theta(\xi)|^2\,d\xi\right)-1\right]^{1/2}=0.
\end{align*}
This proves (\ref{equallaw}).
\end{proof}

\begin{proof}[Proofs of Theorem \ref{thm:main} and Proposition \ref{prop:main}]
Since the nonexplosivity of $\tilde{z}_t$ is proved in Proposition \ref{prop:nonexplosive}, we concentrate on the ergodic theorem. We already showed in Lemma \ref{lem:uniqueinv} that, under $P$, the dynamics of $\tilde{z}_t$ converges to the measure $\pi$. We now strengthen this convergence to the $P$-a.s.\@ sense: for every $f\in L^1(\pi)$,
\begin{align}
\label{eq:ztildeergodic}
P\left(\lim_{t\to\infty} \frac{1}{t} \int_0^t\, f(\tilde{z}_s)\,ds = \int_{\mathbb{C}} \, f(\tilde{z}) \,d\pi(\tilde{z})\right) =1.
\end{align}
Our approach here is to exploit the equivalence of the probability measures $P$ and $Q$ on $\mathcal{F}_\infty$ [Item (3) of Lemma \ref{lem:timechange}], as well as the equivalence of the $Q$-law of $\tilde{z}_t$ and the $P$-law of $z_t$ [Item (4) of Lemma \ref{lem:timechange}]. Using these two observations, we have that for every Borel measurable subset $A$ of $\mathbb{R}$,
\begin{align*}
P\left(\lim_{t\to\infty} \frac{1}{t} \int_0^t\, f(\tilde{z}_s)\,ds \in A \right) =0  &\Longleftrightarrow Q\left(\lim_{t\to\infty} \frac{1}{t} \int_0^t\, f(\tilde{z}_s)\,ds \in A\right) =0 \\
& \Longleftrightarrow P\left(\lim_{t\to\infty} \frac{1}{t} \int_0^t\, f(z_s)\,ds \in A\right) =0.
\end{align*}
By (\ref{eq:ergo2}), we deduce that
\begin{align*}
P\left(\lim_{t\to\infty} \frac{1}{t} \int_0^t\, f(\tilde{z}_s)\,ds \in A \right) =P\left(\lim_{t\to\infty} \frac{1}{t} \int_0^t\, f(z_s)\,ds \in A\right) =0  
\end{align*}
unless $\int_{\mathbb{C}}\, f(z)\,\pi(dz) \in A$. This implies (\ref{eq:ztildeergodic}). 

Referring back to the notation $\tilde{z}$ and $\tilde{w}$ introduced immediately prior to Proposition \ref{prop:main}, let $\Pi$ be the probability measure on $\mathbb{C}^2$ defined by $\Pi(\tilde{z}, \tilde{w}) = \pi(\tilde{z}) \delta_0(\tilde{w})$, where $\delta_0$ is the delta measure. We are going to show that for all $g\in L^1(\Pi)$,
\begin{align}
\label{ergot1}
P\left(\lim_{t\to\infty} \frac{1}{t} \int_0^t\, g(\tilde{z}_s, \tilde{w}_s) \,ds = \int_{\mathbb{C}^2}\, g(\tilde{z},\tilde{w}) \,d\Pi(\tilde{z}, \tilde{w})\right) =1.
\end{align}
Since the space $C_c(\mathbb{C}^2)$ of continuous functions with compact support is dense in $L^1(\Pi)$, it suffices to prove (\ref{ergot1}) for all $g\in C_c(\mathbb{C}^2)$. Using that $\tilde{w}_s= \tilde{w}_0 e^{-\nu s} \to 0$ as $s\to\infty$, as well as the continuity of $g$, we see that for every $\epsilon>0$, there exists a $\kappa>0$ such that if $\|(\tilde{z}_s,\tilde{w}_s)-(\tilde{z}_s, 0)\| =|\tilde{w}_s|<\kappa$, then $|g(\tilde{z}_s, \tilde{w}_s)-g(\tilde{z}_s,0)|<\epsilon$. Fix an $r>0$ such that $|\tilde{w}_0|e^{-\nu r} \leq \kappa$. Then
\begin{align*}
\frac{1}{t} \left|\int_0^t\, \left(g(\tilde{z}_s, \tilde{w}_s) - g(\tilde{z}_s,0)\right)\,ds\right| &\leq\frac{1}{t} \left[\int_0^r\, |g(\tilde{z}_s, \tilde{w}_s) - g(\tilde{z}_s, 0)|\,ds + \int_r^t\, |g(\tilde{z}_s, \tilde{w}_s) - g(\tilde{z}_s, 0)|\,ds \right]\\
& < \frac{r}{t} \|g\|_\infty + \frac{t-r}{t} \epsilon.
\end{align*}
Taking the limsup as $t\to\infty$ on both sides yields
\begin{align*}
\limsup_{t\to\infty}\frac{1}{t} \left|\int_0^t\, \left(g(\tilde{z}_s, \tilde{w}_s) - g(\tilde{z}_s,0)\right)\,ds\right|<\epsilon.
\end{align*}
Since $\epsilon>0$ is arbitrary, and the limit $\lim_{t\to\infty} \frac{1}{t}\int_0^t\, g(\tilde{z}_s,0)\,ds$ exists $P$-a.s., we deduce that
\begin{align}
\label{ergot2}
\lim_{t\to\infty} \frac{1}{t} \int_0^t\, g(\tilde{z}_s, \tilde{w}_s)\,ds = \lim_{t\to\infty} \frac{1}{t} \int_0^t \, g(\tilde{z}_s, 0)\,ds \qquad P\text{-a.s.}
\end{align}
Meanwhile, by (\ref{eq:ztildeergodic}) and the definition of $\Pi$,
\begin{align}
\label{ergot3}
\lim_{t\to\infty} \frac{1}{t}\, \int_0^t g(\tilde{z}_s, 0)\,ds = \int_{\mathbb{C}}\, g(\tilde{z},0) \,d\pi(\tilde{z}) = \int_{\mathbb{C}^2}\, g(\tilde{z}, \tilde{w})\, d\Pi(\tilde{z},\tilde{w})  \qquad P\text{-a.s.}
\end{align} 
Putting (\ref{ergot2}) and (\ref{ergot3}) together yields (\ref{ergot1}). 

We have thus proved that the system of time-homogeneous SDEs (\ref{eq:MainSDE}) converges to a unique ergodic measure $\Pi$. 
\end{proof}


\section{Numerical Results} \label{sec:numerical}

In this section we provide a numerical perspective for solving our stabilization by noise problem, and expand upon the analysis conducted in previous sections.

\subsection{Nonexplosivity upon stabilization by Brownian noise}\label{sec:grid}
We have shown that a necessary condition for the deterministic system (\ref{eq:1}) to admit finite-time blow-up solutions is when $y_3(t) = 0$ for all $t\geq0$ (Remark \ref{rem:y3always0}). Hence, to stabilize this system, we add a Brownian noise which ensures that $y_3(t)\neq0$ for all $t\geq0$. Our simulations, described below, suggest that it is enough to add a real-valued Brownian noise in the ${\rm Im}(z)$ direction; that is,  the corresponding SDE reads
\begin{equation}\label{eq:ou}
    \left\{
                \begin{array}{ll}
                  dx_1 = (-\nu x_1 + \alpha (x_1x_3-x_2x_4))dt\\
                  dx_2 = (-\nu x_2 + \alpha (x_1x_4+x_2x_3))dt + dB_t\\
                 dx_3 = (-\nu x_3 + \beta (x_1x_3-x_2x_4))dt\\
                 dx_4 = (-\nu x_4 + \beta (x_1x_4+x_2x_3))dt\\
                \end{array}
              \right.
\end{equation}
in the ${\bf x}$ coordinates, or
%
%
\begin{equation}\label{eq:14}
    \left\{
                \begin{array}{ll}
                  dy_1=(-\nu y_1+\beta[(y_1^2-y_2^2)-(y_3^2-y_4^2)])dt \\
                 dy_2=(-\nu y_2)dt\\
                 dy_3=(-\nu y_3+2\beta(y_1y_3-y_2y_4))dt + \frac{1}{2}dB_t\\
                 dy_4=(-\nu y_4)dt + \frac{1}{2}dB_t \\
                \end{array}
              \right.
\end{equation} 
in the ${\bf y}$ coordinates.

Observe that if we complexify the coordinates in (\ref{eq:14}) by taking $\tilde{w}_t=y_2(t)+iy_4(t)$, then $\tilde{w}_t$ satisfies the SDE $d\tilde{w}_t=-\nu\tilde{w}\,dt+\frac{i}{2}\,dB_t$, which is a 2-dimensional Ornstein-Uhlenbeck process. (Compare this against the choice of additive Brownian noise in (\ref{eq:MainSDE}), where $\tilde{w}_t$ satisfies the deterministic equation $d\tilde{w}_t  = -\nu\tilde{w}\,dt$.) Since the Ornstein-Uhlenbeck process is ergodic and has an explicit invariant measure, we believe that the system (\ref{eq:ou}) should be nonexplosive, and may be ergodic. As of this writing, we are not in a position to prove these statements, due to some technicality involved in carrying out a time change similar to the one done in \S\ref{sec:ergo}.

\begin{figure}
\centering
\renewcommand{\arraystretch}{1.2}
\begin{small}
\begin{tabular}[b]{ccc}
 & Without added noise & With added noise \\
\rotatebox{90}{$\qquad \qquad y_2(0)=y_3(0)=0$} &
 \includegraphics[width=0.38\textwidth]
    {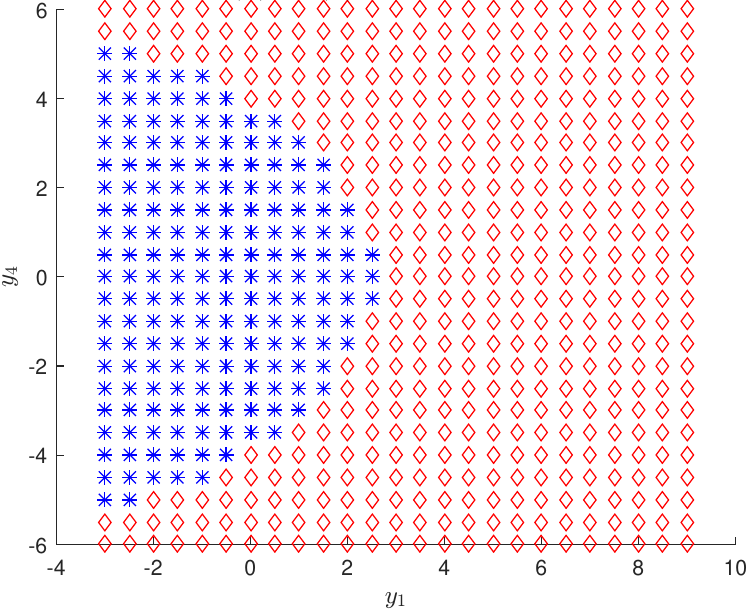}
    &
  \includegraphics[width=0.38\textwidth]
    {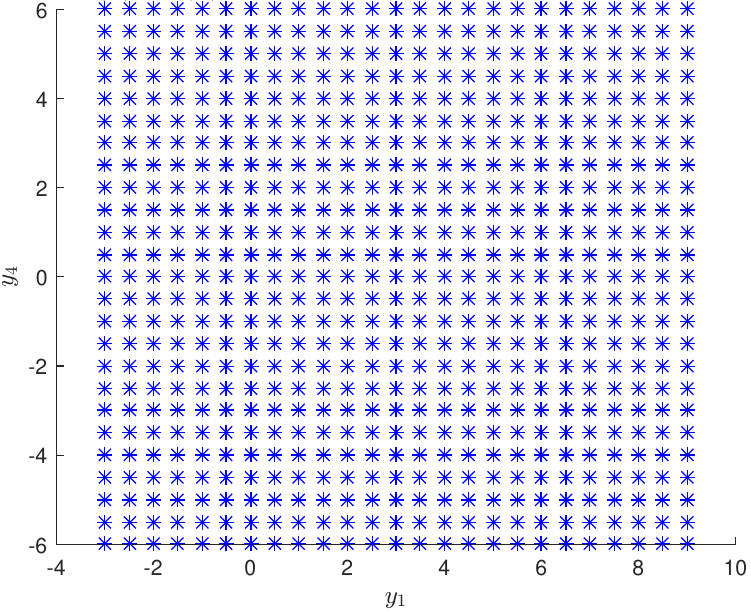}  \\ 
      \rotatebox{90}{$\qquad \qquad y_3(0)=y_4(0)=0$}
    &
   \includegraphics[width=0.38\textwidth]
    {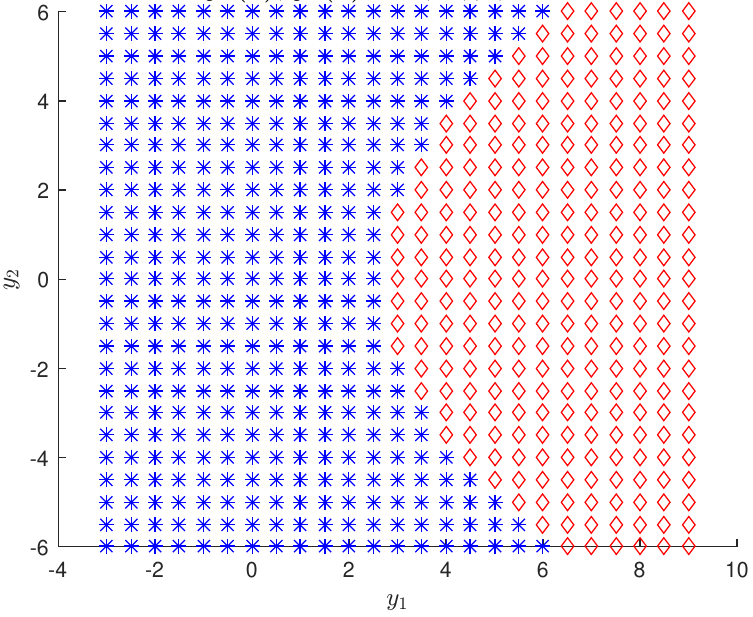}
    &
   \includegraphics[width=0.38\textwidth]
    {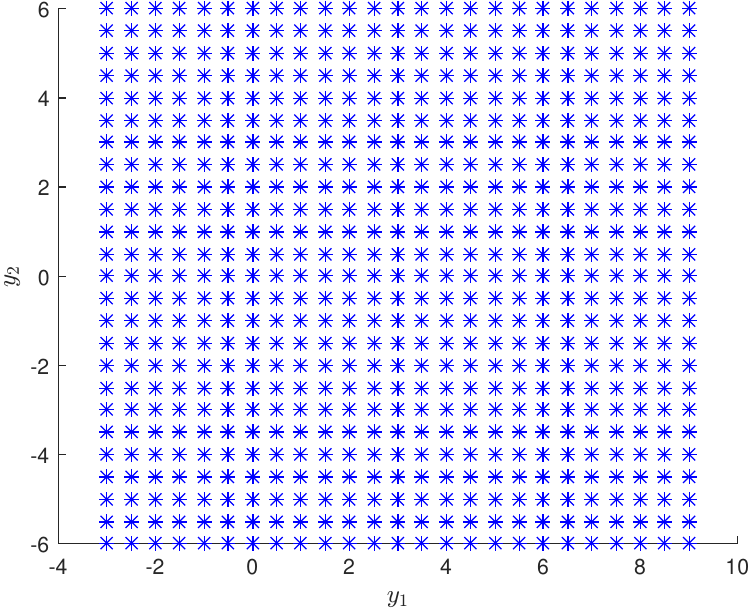} \\
    \rotatebox{90}{ $\qquad \quad y_3(0)=0,~y_4(0)=1$}
    &
  \includegraphics[width=0.38\textwidth]
    {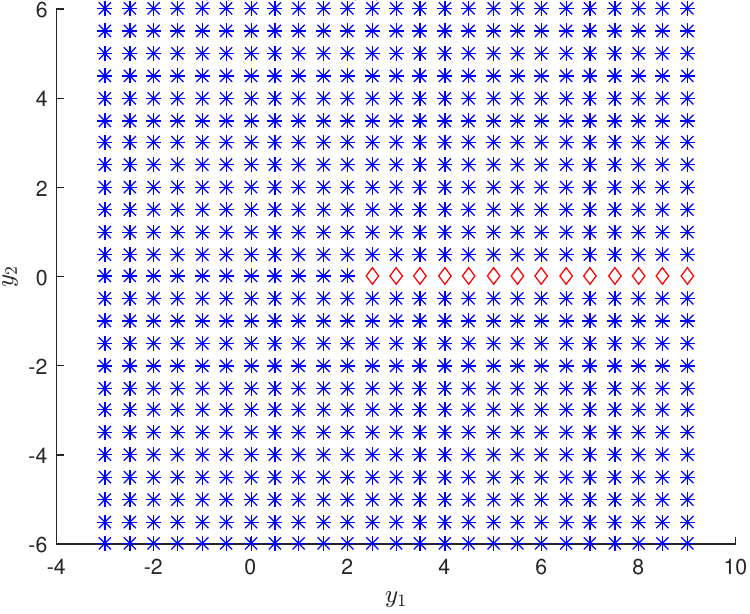}
    &
    \includegraphics[width=0.38\textwidth]
    {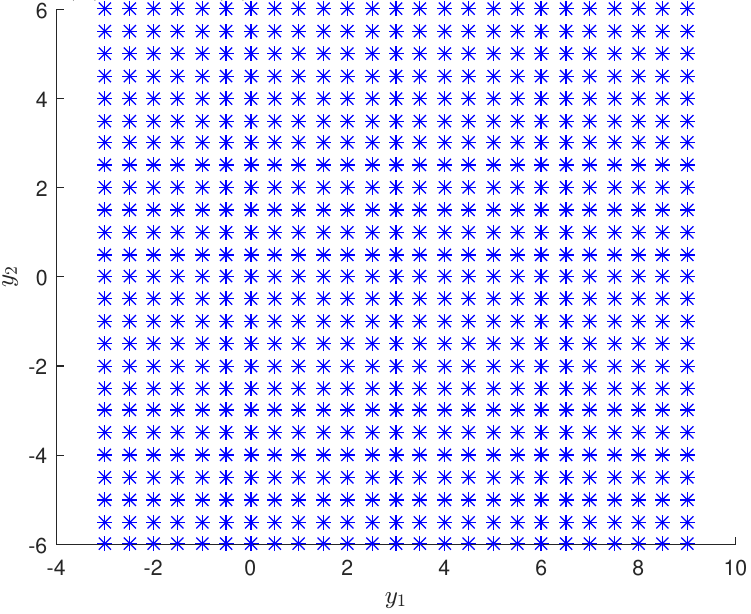}\\
     \rotatebox{90}{$\qquad \qquad y_3(0)=y_4(0)=1$}
    &
   \includegraphics[width=0.38\textwidth]
    {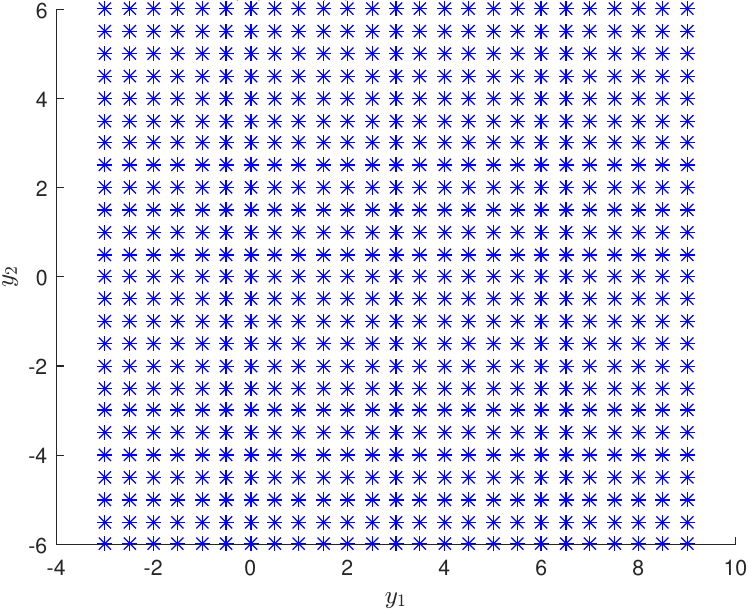} 
    &
    \includegraphics[width=0.38\textwidth]
    {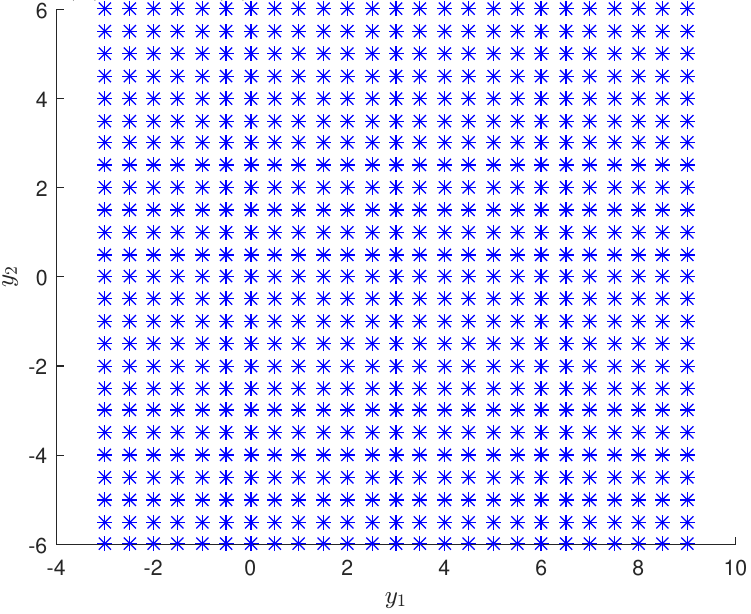}
\end{tabular}
\end{small}
\caption{Simulation results of various initial conditions which give rise to nonexplosive solutions (indicated by {\color{blue} $\ast$}) and explosive solutions (indicated by {\color{red} $\diamondsuit$}) of our $\mathbb{C}^2$-valued coupled system (\ref{eq:1}), with $\beta=\nu=1$. The results on the left panel are for the system without added Brownian noise, while the results on the right panel are with added Brownian noise of the form (\ref{eq:14}).}
\label{fig:sim}
\end{figure}

That said, we have numerical evidence for nonexplosivity of system (\ref{eq:14}). We first simulated the trajectories of the ODE (\ref{eq:5}) (without noise). Using MATLAB, we created a function that takes in a set of initial conditions $y_1(0), y_2(0), y_3(0), y_4(0), \alpha,\beta,$ and $\nu$, and produces a discrete-time approximate solution of (\ref{eq:5}) via Euler's method. We then created a program that fixes two of the initial coordinates (for instance, $y_3(0)=y_4(0)=0$) and $\alpha, \beta$, and $\nu$, while varying the other two coordinates. The program then ran a series of simulations, testing every set of inputs in a grid of the varying initial coordinates. Thus, this program examined a two dimensional grid in the 4-dimensional space of the ODE. For each set of initial conditions, our program recorded whether or not that trajectory blows up in a set period of time. It then recorded the result in a 2-dimensional plot (whose axes are the two varying initial conditions). Our MATLAB code is available for download at \cite{matlab}.

On the left panel in Figure \ref{fig:sim} we present our results of simulating trajectories where $\beta=\nu=1$, and two of the four coordinates are initially fixed, while the other two are being varied. The red diamonds indicate initial conditions which lead to finite-time blow-up trajectories, while blue stars indicate those that give rise to nonexplosive trajectories. Observe that without added noise, our simulations for the fixed initial conditions $y_3(0)=y_4(0)=0$  and $y_2(0)=y_3(0)=0$ correspond, respectively, with the phase portraits in Figures \ref{fig:4} and \ref{fig:5}. Also, for the fixed initial condition $y_3(0)=0$ and $y_4(0)=1$ (without added noise), the explosive trajectory lies on the line $y_2=0$. Contrast this with the fixed initial condition $y_3(0)=y_4(0)=1$, which do not give rise to explosive trajectories. Our numerical simulations verify the analysis in \S\ref{sec:analysis}. 

We used the same procedure to verify our analysis of the SDE (\ref{eq:14}). We ran this program again, this time with our trajectory function programmed with Brownian noise added to $y_3$ and $y_4$. The Brownian noise is modeled by a normally distributed random variable, scaled by the square root of the time step, to each step of the iterated Euler's method. Then we ran the simulations and generate the explosive and nonexplosive initial conditions as before; see the right panel in Figure \ref{fig:sim}. It appears evident that the SDE (\ref{eq:14}) is stable globally. We ran this computation many times on the same set of initial conditions to ensure that the probability of a stable trajectory is near $1$. 

While not shown in Figure \ref{fig:sim}, we have tested a large variety of initial conditions to ensure that SDE (\ref{eq:14}) is stable everywhere in the 4-dimensional space, for all values of $\nu>0$ and $\alpha, \beta\in \mathbb{R}$. 


%

\subsection{Invariant measure}\label{sec:im}

In \S\ref{sec:ergo}, we proved that the system (\ref{eq:ztilde}) has a unique invariant measure. However, characterizing this invariant measure analytically is challenging, so we take a numerical approach here.

Consider the SDEs
\begin{equation}\label{eq:sde}
    \left\{
                \begin{array}{ll}
                  dy_1=(-\nu y_1+\beta[(y_1^2-y_2^2)-(y_3^2-y_4^2)])\,dt + \sigma_1 \,dB_t^1\\
                 dy_2=(-\nu y_2)\,dt + \sigma_2 \,dB_t^1\\
                 dy_3=(-\nu y_3+2\beta(y_1y_3-y_2y_4))\,dt + \sigma_3 \,dB_t^2\\
                 dy_4=(-\nu y_4)\,dt +\sigma_4 \,dB_t^2 ,\\
                \end{array}
              \right.
\end{equation}

where \[ \sigma =  \left( \begin{array}{cc}
\sigma_1 & 0 \\
\sigma_2 & 0\\
0 & \sigma_3\\
0 & \sigma_4 \end{array} \right).\]
We will consider only the case $\sigma_2=\sigma_4=0$ that corresponds to adding an isotropic Brownian noise, which was the case analyzed in \S\ref{sec:ergo}.

In \S\ref{sec:ergo}, we proved that the system (\ref{eq:ztilde}) has the same invariant measure as the system (\ref{eq:z}). To compute for the invariant measure of the system (\ref{eq:ztilde}), we will compute the invariant measure for the system
\begin{equation}\label{eq:2d}
    \left\{
                \begin{array}{ll}
                  dy_1 = (-\nu y_1 + \beta(y_1^2-y_3^2))\,dt + \sigma_1\, dB_t^1\\
                 dy_3 = (-\nu y_3 + 2\beta y_1y_3)\, dt + \sigma_3\, dB_t^2.\\
                \end{array}
              \right.
\end{equation}



To find the invariant measure for (\ref{eq:2d}), we solve the following non-elliptic PDE
\begin{equation}
    \left\{
                \begin{array}{ll}
                 \mathcal{L}^* f = 0\\
		f(y_1,y_3) \rightarrow 0 \quad \text{as} ~||(y_1,y_3)|| \rightarrow +\infty ,
                \end{array}
              \right.
\end{equation} 
where $f=\frac{d\pi}{d\lambda}$, $\lambda$ is the $2-$dimensional Lebesgue measure, and the $\mathcal{L}^*$, the adjoint of $\mathcal{L}$, is given in this case by
 $$\mathcal{L}^*=-\partial_{y_1}((-\nu y_1 + \beta y_1^2-\beta y_3^2)(\cdot))-\partial_{y_3}((-\nu y_3 + 2 \beta y_1 y_3)(\cdot)) + \frac{1}{2}(\sigma_1^2\partial_{y_1y_1}+\sigma_3^2\partial_{y_3y_3}).$$
This gives the steady-state solution to the forward Kolmogorov (or Fokker-Planck) equation associated with the SDE (\ref{eq:2d}). We employ the MATLAB PDE Toolbox to solve this PDE using the finite-element method. We approximate the solutions by solving
\begin{equation}
    \left\{
                \begin{array}{ll}
                 \mathcal{L}^* f = 0 &\quad \text{in}~B_4(0)\\
		f(y_1,y_3) = .1 &\quad \text{on} ~\partial B_4(0) 
                \end{array}
              \right.,
\end{equation} 
where $B_4(0)$ is the ball of radius 4 centered at the origin. The size of the boundary conditions and radius of the ball are mostly irrelevant. Altering them would roughly be equivalent to rescaling the units of the resulting measure. 

After a number of mesh refinements, we obtain an approximate density plot for the invariant measure for (\ref{eq:2d}), see Figure \ref{fig:iso}. Observe that the invariant measure has a peak which is symmetric about $y_3=0$, and slightly skewed toward the left half plane ($y_1<0$). Also, the measure appears to have a heavy-tailed distribution (higher moments may be infinite), which is consistent with the result of Herzog and Mattingly \cite{noise1} in their analysis of (\ref{eq:poly}).

\begin{figure}
\centering
\includegraphics[width=0.45\textwidth]{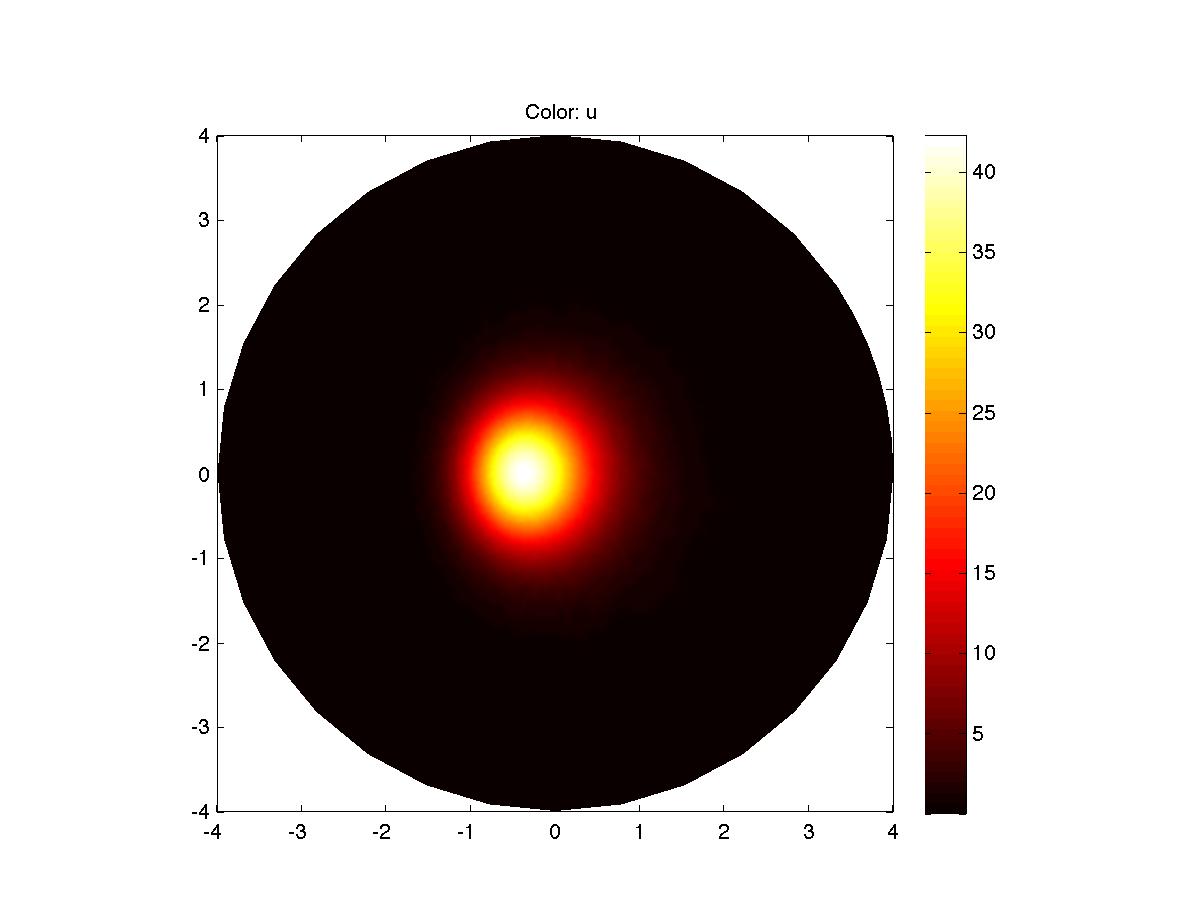}
\caption{Intensity plot for the invariant measure for $\beta=\nu=1$.}
\label{fig:iso}
\end{figure}


\section{Conclusion} \label{sec:conclusion}
Extending existing research on the stabilization of $\mathbb{C}$-valued polynomial ODEs \cite{2011,noise1}, we have ascertained that the addition of a Brownian noise to our prototype multivariable system of ODEs stabilizes explosive (and thus all) trajectories with probability one. This may be seen as a first step toward understanding higher-dimensional stochastic Burgers' equations \cite{transition}, as well as higher-dimensional analogs of complex Langevin equations studied by Aarts \emph{et al.} \cite{Aarts1, Aarts2}.

While we have analytically and numerically verified conditions for stabilization of our coupled ODEs, there remain many open questions. How would our results differ if we change our ODEs in any of the following manners: (1) make the drift parameter $\nu$ negative; (2) if the drift parameters for the two complex coordinates $\nu_1$ and $\nu_2$ are distinct (in which case it may be difficult to find a coordinate transformation like the one used in this paper)? Additionally, we would like to go beyond our system and consider the stabilization problem in more general nonlinear systems. For instance, would our methods still apply to systems in higher dimensions, say in $\mathbb{C}^n$, $n\geq 3$? What about coupled ODEs wherein the coupling terms are higher than quadratic order?

\bibliographystyle{amsalpha}
\bibliography{ref}

\end{document}